\documentclass[11pt,reqno]{amsart}
\usepackage{amsfonts}
\usepackage{amsmath}
\usepackage{amssymb}
\usepackage{multicol}
\usepackage{graphics}
\usepackage{stmaryrd}
\usepackage{cite}
\usepackage{extarrows}
\usepackage{color}
\usepackage{mathtools}
\usepackage{amsfonts}
\usepackage{amsmath}
\usepackage{amssymb}
\usepackage{multicol}
\usepackage{stmaryrd}
\usepackage{cite}
\usepackage{epsfig}
\usepackage{color}
\usepackage{graphics}
\usepackage{graphicx}
\usepackage{epstopdf}
\usepackage{multicol,graphics}
\usepackage{enumitem}
\usepackage{comment}
\usepackage{float}  
\usepackage{subfig}
\usepackage{xcolor}

\newcommand\bes{\begin{eqnarray}}
\newcommand\ees{\end{eqnarray}}

\newtheorem{theorem}{Theorem}[section]

\numberwithin{equation}{section}

\allowdisplaybreaks[4]
\begin{document}
\author[Wang]{Wang Shuang-Ming $^{1}$
\\
 \textit{$^{1}$
 School of  Finance, Lanzhou University of Finance}\\
 \textit{and Economics, Lanzhou, Gansu 730020, People's Republic of China }\\
}
\title[the minimal wave speed of diffusive kermack-mckendrick model]
{\textbf{the minimal wave speed of time-periodic traveling waves arising form a 
diffusive kermack-mckendrick model with seasonality and nonlocal delayed interactions} }
\date{}
\maketitle
\begin{abstract}
This paper is concerned with the non-existence of time-periodic traveling wave solution with speed less than the critical speed for diffusive Kermack-McKendrick epidemic model incorporating
seasonality and nonlocal interactions induced by latent period.
By a technical construction of upper and lower solutions on truncated intervals for an auxiliary linear equation, we overcome the challenges arising from the coupling of nonlocal delay and the fact that the system is non-autonomous.
Thus the critical value $c^*$ defined in \cite{i} is confirmed as the minimal wave speed of time-periodic traveling waves.
We have completely solved the open problem \cite{i}.
\\ \textbf{Keywords}:
 Kermack-Mckendrick model,~nonlocal interactions,~ time-periodic traveling waves,~ the minimal wave speed\\
\textbf{AMS Subject Classification (2000)}: 35K57, 35B35,  35B40, 92D30, 93B60
\end{abstract}
\section{Introduction }
\noindent
In this paper, we are concerned with the following time-periodic and diffusive Kermack-Mckendrick model with nonlocal delay:
\begin{equation}\label{RWA}
\begin{cases}
\begin{aligned}
{\partial_t}S(t,x)&=d_1\partial_{xx} S(t,x)- \beta(t)S(t,x)I(t,x),\\
{\partial_t}I(t,x)&= d_2\partial_{xx} I(t,x) \\&
+
\int_\mathbb{R}\Gamma(t,t-\tau;x-y)\beta(t-\tau)S(t-\tau,y)I(t-\tau,y){\rm d}y-\gamma(t)I(t,x),
\end{aligned}
\end{cases}
~(t,x)\in \mathbb{R}^2.
\end{equation}
Model \eqref{RWA} was derived in \cite{i}, wherein $S(t,x)$ and $I(t,x)$ present the densities of susceptible class and infectious class  in time $t$ and location $x$, respectively.
Constant $T>0$ is the periodicity of environmental variation over time and $\tau>0$ is the incubation period of the pathogen.
Positive constants $d_1$ and $d_2$  denote the diffusion rates of susceptible class and infectious class, respectively.
$T$-periodic, positive and continuous functions $\beta$ and $\gamma$ are the infectious coefficient and the removed rate, respectively.
The kernel function 
$$\Gamma(t,t-\tau;y)=e^{-\int_{t-\tau}^{t}\gamma_L(s){\rm d}s}\frac{1}{ \sqrt{4\pi d_L\tau}} e^{-\frac{y^{2}}{4 d_L\tau}},~y\in \mathbb{R}$$
reveals the nonlocal interactions derived by mobility of individuals during latent period of disease, where $\gamma_L(t)$ is also a continuous and $T$-periodic function, $d_L$ is a positive constant.
Then $\Gamma$ is $T$-periodic with respect to the first two independent variables.
For the sake of simplicity, let $\mathcal{G}(t,t-\tau;y)=\frac{1}{ \sqrt{4\pi d_L\tau}} e^{-\frac{z^{2}}{4 d_L\tau}}$,
then we have
$\Gamma(t,t-\tau;y)=e^{-\int_{t-\tau}^{t}\gamma_L(s){\rm d}s}\mathcal{G}(t,t-\tau;y)$.
\noindent

Wang et al. \cite{i} recently studied the existence of time-periodic traveling wave solutions connecting two disease-free states of system \eqref{RWA}.
We recall that the time-periodic traveling wave solution of \eqref{RWA}
is a special solution with the form $(S(t,x),I(t,x)):=(\phi(t,x+ct),\psi(t,x+ct))$, wherein $(\phi,\psi)$
is $T$-periodic with respect to the first independent variable.
Hence, $(\phi,\psi)$ satisfies
\begin{equation}\label{trave}
\begin{cases}
\partial_{t}\phi(t,z)=d_1\partial_{zz}\phi(t,z)- c\partial_{z}\phi(t,z)-\beta(t)\phi(t,z)\psi(t,z),\\
\partial_{t}\psi(t,z)= d_2\partial_{zz}\psi(t,z)- c\partial_{z}\psi(t,z)\\~~~~~
+\int_\mathbb{R}\Gamma(t,t-\tau;y)\beta(t-\tau)
\phi(t-\tau,z-c\tau-y)\psi(t-\tau,z-c\tau-y){\rm d}y-\gamma(t)\psi(t,z)
\end{cases}
\end{equation}
and
\begin{equation}\label{Tp}
 (\phi(t+T,z),\psi(t+T,z))=(\phi(t,z),\psi(t,z))
\end{equation}
for all $(t,z)\in \mathbb{R}^2$.
We note that $c>0$ is the wave speed, $z=x+ct$ is the moving coordinate.
Additionally, to model  the disease propagation governed by \eqref{RWA}, we impose the following asymptotic boundary conditions on $(\phi,\psi)$:
  \begin{equation}\label{BC}
  \phi(t,-\infty)=S_0,\, \phi(t,\infty)=S_{\infty},\, \psi(t,\pm\infty)=0 ~\text{uniformly for all}~t\in \mathbb{R},
\end{equation}
wherein $S_0>0$ is the initial density of susceptible class, while $S_\infty>0$ is the density of susceptible class after the epidemic.
In fact, $(S_0,0)$ and $(S_\infty,0)$ are two equilibriums of kinetic system
\begin{equation}\label{ODE}
\begin{cases}
 S'(t)=- \beta(t)S(t)I(t),\\
I'(t)=
e^{-\int_{t-\tau}^{t}\gamma_L(s){\rm d}s}\beta(t-\tau)S(t-\tau)I(t-\tau)-\gamma(t)I(t).
\end{cases}
\end{equation}
Inspired by \cite{Ruan07}, we see that such time-periodic traveling waves for epidemic models implies a moving zone of transition from an infective state to a disease-free state in time-periodic environment.
 It also establishes a rigorous mathematical basis for studying the wave-like propagation of diseases in seasonally modulated landscapes.
\noindent

In actuality, time-periodic traveling wave solutions for periodic  SI epidemic systems have been widely studied in recent years.
Zhang et al. \cite{z-KM,DengD22} investigated  the existence of time-periodic traveling waves of Kermack-Mckendrick model with time-periodic coefficients.
The pivotal methodology involves  a novel two-step utilization of fixed-point theorems.
Such idea has also served as the foundational framework for addressing analogous issues in  further studies \cite{CPAA,ZhaoL,Z,SAM}.
Relevant results on high-dimensional systems can also be found in \cite{SAM,YXY,ZhangWu,D21}, wherein \cite{ZhangWu} concurrently accounts for discrete time delay.
Regardless of either high-dimensional models or delayed models, the critical wave speed is implicitly defined rather than explicitly formulated.
In both cases, the nonexistence of time-periodic traveling waves with wave speeds below the minimum wave speed is typically established by utilizing the asymptotic spreading speed properties\cite{SAM,ZhangWu}.
\noindent

To reveal the spatiotemporal propagation phenomena caused by the alternation of seasons, the authors\cite{i}  defined a critical value $c^*$ and proved the existence of time-periodic traveling waves of  \eqref{RWA}  with any given speed $c>c^*$. 
Nevertheless, \cite{i} did not prove that \eqref{RWA} admits no time-periodic traveling waves with
speed $c\in (0,c^*)$.
In other words, \cite{i} did not confirm the critical value $c^*$ is the minimal wave speed of time-periodic traveling waves.
In addition, \cite{i} shows the impacts of strength of seasonal forcing on $c^*$ and
obtained some new meaningful epidemiological insights through numerical simulations.
It is essential to note that these conclusions hold only if  $c^*$ coincides with the minimal wave speed.
\noindent

However, it is quite challenging to prove the non-existence of time-periodic traveling waves with speed $c\in (0,c^*)$ because of the  fact that system is non-autonomous and the introduce of nonlocal kernel function.
Firstly, the Laplace transform, which is applicable to autonomous systems(including time-delayed cases), fails for non-autonomous systems.
Secondly, the asymptotic spreading speed results, as employed in \cite{Z,SAM,ZhangWu,Z25JDE}, also fail to hold for system \eqref{RWA} due to the fundamental differences in the population dynamical structures.
It is worth noting that Ducrot et.al\cite{Ducrot2011} introduced a comparison approach
via the construction of a linearized equation with small perturbation on a sufficiently long truncated interval.
Soon afterwards, the development of time-periodic solutions to such linearized equation allowed for the extension of this methodology to time-periodic systems\cite{z-KM,CPAA}.
However, when nonlocal delay is incorporated,  the presence of $z-c\tau-y$  in \eqref{trave} prevents the formation of truncated interval and time-periodic solution for comparison purpose.
It should also be noted that spread problems involving nonlocal delays are inherently more difficult than its discrete counterpart\cite{ZhangWu}.
\noindent

In this paper, we intend to show the non-existence of time-periodic traveling waves of \eqref{BC} with speed $c\in (0,c^*)$ and boundary condition \eqref{BC}, and  thereby  confirm $c^*$ is the minimal wave speed of time-periodic traveling waves.
To address aforementioned challenges, we introduce an auxiliary linear equation with a significant perturbation and then construct upper and lower solutions carefully on truncated intervals for such auxiliary equation.
The intended objective is thereby achieved by deriving a contradiction through the comparison of auxiliary equation.
The remaining part of this paper provides the main result and its proof.
\section{The nonexistence of periodic traveling waves}
\noindent

We first recall the definition of  the basic reproduction number $R_0$ (\cite{i}).
Linearizing the second equation of  \eqref{ODE} at the disease-free steady state $(S _0,0)$  and applying $\varepsilon$-perturbation to it, we have
\begin{equation}\label{Linear}
 I'(t)=  e^{-\int_{t-\tau}^{t}\gamma_L(s){\rm d}s}\beta(t-\tau)(S_0-\varepsilon)I(t-\tau)-\gamma(t)I(t).
\end{equation}
We can define the basic reproduction number of linear equation \eqref{Linear} and denote it by $R_0^\varepsilon$(\cite{zhao17}).
In addition, we observe that $\lim_{\varepsilon\rightarrow0} R_0^\varepsilon=R_0$ and $\lim_{\varepsilon\rightarrow S_0} R_0^\varepsilon=0$.
Then $\varepsilon_{sup}:=\sup\{ \varepsilon\in(0,S_0):~R_0^\varepsilon>1\}$ is well defined provided $R_0>1$.
\noindent
We now present and prove the main result of this study.
\begin{theorem}\label{mt}
Assume that $R_0>1$. For all $c\in(0,c^{*})$,
system \eqref{RWA} admits no time-periodic traveling wave solution with wave speed $c$ that satisfies
asymptotic boundary condition \eqref{BC}.
\end{theorem}
\begin{proof}
Suppose, by contradiction, that there exists some  $c\in(0,c^{*})$  such \eqref{trave} have a solution $(\phi(t,x+ct),\psi(t,x+ct))$  satisfying \eqref{Tp} and \eqref{BC}.
We begin by identifying an $\ell>0$  such that
\begin{equation}\label{BiaoJi}
\int_{-\ell}^{\ell} \frac{1}{ \sqrt{4\pi d_L\tau}} e^{-\frac{y^{2}}{4 d_L\tau}}{\rm d}y=1-\varrho:=\frac{1}{A}.
\end{equation}
By virtue of the continuous dependence of the solutions and the basic reproduction number on the coefficients, we can
choose an $\varepsilon_*\in(0,\varepsilon_{sup})$ such that $R_0^{\varepsilon_*}>1$ and
\begin{equation}\notag
u'(t)=e^{-\int_{t-\tau}^{t}\gamma_L(s){\rm d}s}\beta(t-\tau)(S_0-\varepsilon_*)\frac{u(t-\tau)}{1+u(t-\tau)}-\gamma(t)u(t)
\end{equation}
possesses a unique positive and $T$-periodic solution $u_{\varepsilon_*}$ satisfying $\max_{t\in[0,T]}u_{\varepsilon_*}(t)< A-1$.
\noindent

Constructing the following equation
\begin{equation}\label{CLE}
\begin{aligned}
&\partial_t U(t,z)-d_2\partial_{zz}U(t,z) +c\partial_z U(t,z)\\
&-e^{-\int_{t-\tau}^{t}\gamma_L(s){\rm d}s}\int_{-\ell}^{\ell}\mathcal{G}(t,t-\tau;y)\beta(t-\tau)(S_0-\varepsilon_*)
U(t-\tau,z){\rm d}y+\gamma(t) U(t,z)=0.
\end{aligned}
\end{equation}
Moreover, it follows from \eqref{BC} that there exits some $r>>\ell$  and $l>\ell^2$ 
such that
$$\phi(t-\tau,z-c\tau-y)>S_0-\varepsilon_*
$$
$$ u_{\varepsilon_*}(t-\tau)+\psi(t,z)
\leq
 \frac{A u_{\varepsilon_*}(t-\tau) }{1+u_{\varepsilon_*}(t-\tau)}
 +\psi(t,z-c\tau-y)$$
for all $y\in[-\ell,\ell]$ and $z\in[-r-l,-r]$.
Denote $W(t,z):=u_{\varepsilon_*}(t)+\psi(t,z)$,  we have
\begin{align}\notag
&\quad\quad\partial_t W(t,z)-d_2\partial_{zz}W(t,z) +c\partial_z W(t,z)\\  \notag
&-e^{-\int_{t-\tau}^{t}\gamma_L(s){\rm d}s}\int_{-\ell}^{\ell}\mathcal{G}(t,t-\tau;y)\beta(t-\tau)(S_0-\varepsilon_*)
W(t-\tau,z){\rm d}y+\gamma(t) W(t,z)\\  \notag
&=u_{\varepsilon_*}'(t)+\partial_t\psi(t,z)-d_2\partial_{zz}\psi(t,z) +c\partial_z\psi(t,z)\\ \notag
&-e^{-\int_{t-\tau}^{t}\gamma_L(s){\rm d}s}\!\!\int_{-\ell}^{\ell}\!\mathcal{G}(t,t-\tau;y)\beta(t-\tau)(S_0\!-\!\varepsilon_*)
[u_{\varepsilon_*}(t-\tau)+\psi(t-\tau,z)]{\rm d}y+\gamma(t) [u_{\varepsilon_*}(t)+\psi(t,z)]\\ \notag
&\geq  u_{\varepsilon_*}'(t)+\partial_t\psi(t,z)-d_2\partial_{zz}\psi(t,z) +c\partial_z\psi(t,z)\\ \notag
&\!-\!e^{-\int_{t\!-\!\tau}^{t}\!\gamma_L(s)
{\rm d}s}\!\!\int_{-\ell}^{\ell}\!\!\mathcal{G}\!(t,t\!-\!\tau;y)\beta(t\!-\!\tau)(S_0\!-\!\varepsilon_*)
\Big[\frac{A u_{\varepsilon_*}(t-\tau) }{1+u_{\varepsilon_*}(t-\tau)}   \!+\!\psi(t\!-\!\tau,z\!-c\tau\!-\!y)\Big]{\rm d}y\!+\!\gamma(t) [u_{\varepsilon_*}(t)\!+\!\psi(t,z)]\\ \notag
&=  u_{\varepsilon_*}'(t)-e^{-\int_{t-\tau}^{t}\gamma_L(s){\rm d}s}\beta(t-\tau)(S_0-\varepsilon_*)
\frac{ u_{\varepsilon_*}(t-\tau) }{1+u_{\varepsilon_*}(t-\tau)}
+\gamma(t)u_{\varepsilon_*}(t)\\    \notag &+\partial_t\psi(t,z)-d_2\partial_{zz}\psi(t,z) +c\partial_z\psi(t,z)\\ \notag &
-e^{-\int_{t-\tau}^{t}\gamma_L(s){\rm d}s}
\int_{-\ell}^{\ell}\mathcal{G}(t,t-\tau;y)\beta(t-\tau)(S_0-\varepsilon_*)
\psi(t-\tau,z-c\tau-y){\rm d}y
+\gamma(t) \psi(t,z)\\  \notag
&> u_{\varepsilon_*}'(t)-e^{-\int_{t-\tau}^{t}\gamma_L(s){\rm d}s}\beta(t-\tau)(S_0-\varepsilon_*) \frac{ u_{\varepsilon_*}(t-\tau) }{1+u_{\varepsilon_*}(t-\tau)
+\gamma(t)u_{\varepsilon_*}(t)}\\  \notag   &+\partial_t\psi(t,z)-d_2\partial_{zz}\psi(t,z) +c\partial_z\psi(t,z)\\ \notag &
-e^{-\int_{t-\tau}^{t}\gamma_L(s){\rm d}s}
\int_\mathbb{R}\mathcal{G}(t,t-\tau;y)\beta(t-\tau)\phi(t-\tau,z-c\tau-y)
\psi(t-\tau,z-c\tau-y){\rm d}y
+\gamma(t) \psi(t,z)\\ \label{UP}
&=0.
\end{align}
\noindent

Let
$\omega(x)=
\sin \big(-\frac{\pi (x+r)}{l}\big)$, $x\in[-r-l,-r]$.
It is easy to see that $\omega''(x)=-\frac{\pi^2}{l^2}\omega(x)$ and $\omega(x)=0$, $x=-r-l,-r$.
In addition, it follows from \eqref{BiaoJi} that
$$\int_{-\ell}^{\ell}\mathcal{G}(t,t-\tau;y)\omega(z){\rm d}y=(1-\varrho)\omega(z).$$
\noindent

Let $\mu_c:=\frac{c}{2d_2}$ and consider equation
\begin{equation}\label{Evpb}
\eta'(t)= [d_2\mu_c^2-\gamma(t)]\eta(t) +e^{-\int_{t-\tau}^{t}\gamma_L(s){\rm d}s}\beta(t-\tau)(S_0-\varepsilon_*)(1-\varrho)\eta(t-\tau).
\end{equation}
Denote $\lambda_c:=\frac{\ln \rho_{\varepsilon_*}(\mu_c)}{T}$, wherein $\rho_{\varepsilon_*}(\mu_c)$ is the spectral radius of Poincar\'{e}  map $P_{\varepsilon_*}$ generated by \eqref{Evpb}.
Let $m:=\min\{k\in\mathbb{N}_+:~k\lambda_c>c\mu_c\}$, which implies $c\mu_c<m\lambda_c$.
Using a method similar as the proof of \cite[Proposition 2.1]{XuDS}, we can obtain
\begin{equation}\label{Kt0}
K'(t)\!-\![d_2\mu_c\!-\!\gamma(t)]K(t)\!-e^{-\int_{t-\tau}^{t}\gamma_L(s){\rm d}s}\beta(t-\tau)(S_0-\varepsilon_*)(1\!-\!\varrho)e^{-m\lambda_c\tau}K(t-\tau)\\
=\!-m\lambda_c K(t)
\end{equation}
admits a unique $mT$-periodic and positive solution.
\noindent
Define $w(t,x):=e^{\rho t}e^{\mu_c(z-ct)}K(t)\omega(x)$, wherein $\rho\in (c\mu_c,m\lambda_c)$ is a fixed  constant.
Recalling that $\ell$ is independent of given $\varepsilon_*$, we can choose $\ell$ sufficiently large such that $\frac{d_2\pi^2}{l^2}<m\lambda_c-\rho$.
Then we have
\begin{align}\notag
&\quad\quad\partial_t w(t,z)-d_2\partial_{zz}w(t,z) +c\partial_z w(t,z)\\ \notag
&-e^{-\int_{t-\tau}^{t}\gamma_L(s){\rm d}s}\int_{-\ell}^{\ell}\mathcal{G}(t,t-\tau;y)\beta(t-\tau)(S_0-\varepsilon_*)
w(t-\tau,z){\rm d}y+\gamma(t) w(t,z)\\ \notag
&
=\rho w(t,z)-c\mu_cw(t,z)+e^{\rho t} e^{\mu_c (z-ct)}
\bigg\{K'(t) \omega(z)-d_2K(t)\big[\mu_c ^2 \omega(z)+ \omega''(z)+2\mu_c  \omega'(z)\big]
\\ \notag
&\!+\!cK\!(t)\big[\mu_c\omega(z)\!+\!\omega'(z)\big]\\ \notag
&\!-\!e^{-\int_{t-\tau}^{t}\gamma_L(s){\rm d}s+d_L\tau\mu^2}\beta(t-\tau)(S_0-\varepsilon_*)e^{-\rho\tau}K(t-\tau)\int_{-\ell}^{\ell}
\mathcal{G}(t,t-\tau;y)\omega(z){\rm d}y + \gamma(t)K\!(t)\omega(z)\!\bigg\}\\ \notag
&=
\rho w(t,z)-c\mu_cw(t,z)+e^{\rho t} e^{\mu_c (z-ct)}
\bigg\{K'(t) \omega(z)-d_2K(t)\big[\mu_c ^2 \omega(z)+ \omega''(z)+2\mu_c  \omega'(z)\big]
\\ \notag
&\!+\!cK\!(t)\big[\mu_c\omega(z)\!+\!\omega'(z)\big]\\ \notag
&\!-\!e^{-\int_{t-\tau}^{t}\gamma_L(s){\rm d}s}\beta(t-\tau)(S_0-\varepsilon_*)e^{-\rho\tau}K(t-\tau)
(1-\varrho)\omega(z)+\gamma(t)K\!(t)\omega(z)\!\bigg\}\\ \notag
&
=\rho w(t,z)+e^{\rho t} e^{\mu_c (z-ct)}
\bigg\{-d_2 \omega''(z)K(t)+K'(t) \omega(z)-d_2\mu_c ^2 \omega(z)K(t)\\ \notag
&\!-\!e^{-\int_{t-\tau}^{t}\gamma_L(s){\rm d}s}\beta(t-\tau)(S_0-\varepsilon_*)e^{-\rho\tau}K(t-\tau)
(1-\varrho)\omega(z)+\gamma(t)K\!(t)\omega(z)\!\bigg\}\\ \notag
&
\leq\rho w(t,z)+e^{\rho t} e^{\mu_c (z-ct)}
\bigg\{-d_2 \omega''(z)K(t)+K'(t) \omega(z)-d_2\mu_c ^2 \omega(z)K(t)\\ \notag
&\!-\!e^{-\int_{t-\tau}^{t}\gamma_L(s){\rm d}s}\beta(t-\tau)(S_0-\varepsilon_*)e^{-m\lambda_c\tau}K(t-\tau)
(1-\varrho)\omega(z)+\gamma(t)K\!(t)\omega(z)\!\bigg\}\\ \notag
&=\rho w(t,z)+e^{\rho t} e^{\mu_c (z-ct)}
\bigg\{-d_2 \omega''(z)K(t)+\big[K'(t) -d_2\mu_c ^2 K(t)\\ \notag
&\!-\!e^{-\int_{t-\tau}^{t}\gamma_L(s){\rm d}s}\beta(t-\tau)(S_0-\varepsilon_*)e^{-m\lambda_c\tau}(1-\varrho)K(t-\tau)
+\gamma(t)K\!(t)\big]\omega(z)\!\bigg\}\\ \notag
&=\rho w(t,z)+e^{\rho t} e^{\mu_c (z-ct)}
\big[-d_2 \omega''(z) - m\lambda_c \omega(z)\big]K(t)\\  \notag
&=\Big(\frac{d_2\pi^2}{l^2}+\rho-m\lambda_c\Big)w(t,z)\leq0   \notag
\end{align}
for all $t>0$ and $z\in[-r-l,-r]$.
Therefore, $w(t,z)$ is a lower solution of (\ref{CLE}).
In contrast, it follows from \eqref{UP} that $W(t,z)$ is a super solution of (\ref{CLE}).
As a consequence, there exist some $\varsigma(n)>0$ sufficiently small such that
$$W(s,z)\geq \varsigma w(s,z)=
\varsigma e^{(\rho-c\mu_c)s}K(t)e^{\mu_c z} \omega(z),
~~ \forall (s,z)\in [-\tau,0]\times[-r-l,-r].$$
On the other hand, we have
\begin{equation}\notag
W(t,z)\geq 0=\varsigma w(t,z)=\varsigma e^{(\lambda_c-c\mu_c) t}K(t)e^{\mu_c z} \omega(z), ~t>-\tau,  z=-r-l,-r.
\end{equation}
Then it follows from comparison principle that
\begin{equation}\notag
W(t,z)\geq \varsigma w(t,z)=\varsigma e^{(\lambda_c-c\mu_c) t}K(t)e^{\mu_c z} \omega(z)
\end{equation}
for all  $t>0$ and $z\in[-r-l,-r]$.
Since $\rho-c\mu_c>0$, we observe that $W(t,z)\rightarrow\infty$ for all $z\in(-r-l,-r)$ as $t\rightarrow\infty$, that is, $\psi(t,z)\rightarrow\infty$ for all $z\in(-r-l,-r)$ as $t\rightarrow\infty$.
This contradiction completes the proof.
\end{proof}
\section{Conclusions}
In this paper, we completely solved the open problem raised in\cite{i}.
Precisely, we prove the non-existence of time-periodic traveling waves with speed less that a critical value $c^*$ defined in \cite{i}.
Thus we confirm that  $c^*$ is exactly the minimal wave speed of time-periodic traveling waves of diffusive Kermack-McKendrick epidemic model \eqref{RWA}.
It is this conclusion that allows for a meaningful interpretation of the numerical findings presented in \eqref{RWA}.
Precisely, models that neglect the latent period  inherently fail to capture the influence of seasonal variations on transmission dynamics\cite{z-KM,CPAA,ZhaoL}.
Present study, complemented by the numerical results of \eqref{RWA}, serve to rectify this limitation.

Based on Ducrot et.al\cite{Ducrot2011} and more recent developments\cite{D21}, we propose a novel perturbation technique that fundamentally differs from existing  approaches.
These conventional methods rely on applying small perturbations to an auxiliary linear equation.
Owing to the presence of nonlocal time delays in current system, we note that the methodology developed in this work is expected to be applicable to similar  problems in time-periodic epidemic systems that account for incubation period.
It also appears to work for estimating the asymptotic spreading speed of solutions with localized initial introductions in time-periodic epidemic systems with incubation period.

\noindent{\bf Acknowledgments}~

S.-M. Wang was partially supported by the National Natural Science Foundation of China [12461034],
the Provincial Talent Program of Gansu Province of China[2025QNTD15]
and the Science and Technology Plan Foundation of Gansu Province of China [24JRRA1006].

\end{document}